\documentclass[12pt]{amsart}

\usepackage{comment}
\usepackage{float} 

\usepackage{enumitem,graphicx}
\setlist[enumerate]{leftmargin=*, font=\upshape, label=\alph*)} 
\setlist[itemize]{leftmargin=*} 

\usepackage{amssymb,amsthm,amsfonts,amsmath}
\usepackage{booktabs}

\newtheoremstyle{example-style}{5pt}{0pt}{}{}{\scshape}{:}{.5em}{}

\newtheorem{Thm}{Theorem}

\newtheorem{conj}{Conjecture}
\newtheorem{lem}{Lemma}

\newtheorem{cor}{Corollary}
\newtheorem{prop}{Proposition}
\newtheorem{Challenge}{Challenge}

\newtheorem{Speculation}{Speculation}

\newcommand*{\bfrac}[2]{\genfrac{}{}{0pt}{}{#1}{#2}}
\renewcommand{\epsilon}{\varepsilon}

\makeatletter
\let\@@pmod\pmod
\DeclareRobustCommand{\pmod}{\@ifstar\@pmods\@@pmod}
\def\@pmods#1{\mkern4mu({\operator@font mod}\mkern 6mu#1)}
\makeatother

\begin{document}


\title{Distributional irregularities of two invariants related to $\mathbb Q(\zeta_q)$}
\title[Irregular behaviour of two
cyclotomic field invariants]{Irregular behaviour of class numbers
and Euler-Kronecker constants of cyclotomic fields: the log log log  devil at play}
\author{Pieter Moree}
\address{Max-Planck-Institut f\"ur Mathematik,
Vivatsgasse 7, D-53111 Bonn, Germany}
\email{moree@mpim-bonn.mpg.de}
\subjclass[2010]{11N37, 11Y60, 11M20}
\maketitle


\begin{abstract} 
Kummer (1851) and, many years later, Ihara (2005) both posed conjectures on invariants
related to the cyclotomic field $\mathbb Q(\zeta_q)$ with $q$ a prime.
Kummer's conjecture concerns the asymptotic behaviour of the first 
factor of the class number of  $\mathbb Q(\zeta_q)$ and Ihara's the positivity of the Euler-Kronecker constant of
$\mathbb Q(\zeta_q)$ (the ratio of the constant and the residue of
the Laurent series of the Dedekind zeta function $\zeta_{\mathbb Q(\zeta_q)}(s)$ at $s=1$).
If certain standard conjectures in analytic number 
theory hold true, then one can show that both conjectures 
are true for a set of primes of
natural density 1, but false in general. Responsible for this are
irregularities in the distribution of the primes.
\par With this survey we hope to convince the reader that
the
apparently dissimilar mathematical objects studied
by Kummer and Ihara actually display a very similar behaviour.
\end{abstract}

\section{Introduction}
Making conjectures in analytic prime number theory 
is a notoriously
dangerous endeavour\footnote{In fact, the title of this paper 
ends with a question mark. Since it is considered very bad style to have 
it in the title of a paper, this footnote might be a better place. Not putting the
question mark would go against the moral of this paper.}, 
certainly if the basis for this is mostly of numerical nature.
The danger lies in the fact that computers can barely spot 
$\log \log $ terms and are certainly blind to the $\log \log \log$ terms
that frequently occur. The presence of such terms can result in the conjecture
being false on very thin subsequences.
Celebrated examples are  the
$\pi(x)< {\rm Li}(x)$ conjecture and the Mertens conjecture 
that $|\sum_{n\le x}\mu(n)|<\sqrt{x}$ for $n\ge 1$ 
(for notation see Section \ref{standard}).
Both of them are false, but true up to gigantic values of $x$.
A way out of the danger zone 
is to change ``for all" to some 
slightly weaker statement.
However, this requires a substantial theoretical insight into the
conjecture.
\par Here we present two further conjectures (due to Kummer and Ihara respectively) where the phenomena
indicated above also seem to
arise. The final verdict on them is still open but, assuming 
some standard conjectures from analytic number theory, they are
false on some very thin sequences of primes due to irregularities
in the distribution of the primes.
At a first glance the two conjectures look unrelated. However,
they are both connected with the distribution of special $L$-values 
and the results and conjectures we present on them are strikingly
similar\footnote{The similarity was first noted by Andrew Granville, 
see acknowledgment.}. 
\par In the remaining part of the introduction we formulate the
conjectures (after stating some 
background material) and discuss how they are related to special $L$-values.
In the rest of the paper we discuss results and related conjectures.
\par Although results from various papers are mentioned in this survey, 
our main inspiration are Ford, Luca and Moree \cite{FLM} for the Euler-Kronecker constant and
Granville \cite{Gr} for Kummer's conjecture. Euler-Kronecker constants 
for non-quadratic fields were put on the
mathematical map mainly thanks to the efforts of Ihara \cite{I,I2,I3}.
\subsection{The Euler-Kronecker constant for number fields}
For a number field $K$ we can define, for Re$s>1$, the 
\emph{Dedekind zeta function} by
$$
\zeta_K(s)=\sum_{\mathfrak{a}} \frac{1}{N{\mathfrak{a}}^{s}}
=\prod_{\mathfrak{p}}\frac{1}{1-N{\mathfrak{p}}^{-s}}.
$$
Here, $\mathfrak{a}$ runs over the non-zero ideals in ${\mathcal O}_K$, the ring of integers of $K$,
$\mathfrak{p}$ runs over the prime ideals in ${\mathcal O}_K$ and $N{\mathfrak{a}}$ is the norm
of $\mathfrak{a}$.
It is known that $\zeta_K(s)$ can be analytically continued to $\mathbb C - \{1\}$,
and that at $s=1$ it has a simple pole and residue $\alpha_K$. 
The prime ideals having prime norm are
of particular importance as they are the cause for this pole.
\par After a suitable normalisation with gamma factors, one obtains from 
$\zeta_K(s)$
a function ${\tilde \zeta}_K(s)$ 
satisfying the functional equation $${\tilde \zeta}_K(s)={\tilde \zeta}_K(1-s).$$ Since
${\tilde \zeta}_K(s)$ is entire of order 1, one has the following Hadamard product factorization:
\begin{equation}
\label{hadimassa}
{\tilde \zeta}_K(s)={\tilde \zeta}_K(0)e^{\beta_Ks}\prod_{\rho}\left(1-\frac{s}{\rho}\right)e^{\frac{s}{\rho}},
\end{equation}
with $\beta_K\in \mathbb C$ and where $\rho$ runs over the
zeros of $\zeta_K(s)$ in the critical strip.
\par Around $s=1$ we have the Laurent expansion
\begin{equation}
\label{laurent}
\zeta_K(s)=\frac{\alpha_K}{s-1}+c_K+c_1(K)(s-1)+c_2(K)(s-1)^2+\ldots.
\end{equation}
The constant $\gamma_K=c_K/\alpha_K$ is called the {\it Euler-Kronecker constant} in
Ihara \cite{I} and Tsfasman \cite{T}.
In particular, 
we have $c_{\mathbb Q}=\gamma=0.57721566\ldots,$ 
the {\it Euler-Mascheroni constant}, see e.g. Lagarias \cite{Lagarias} for a wonderful survey of 
related material. In case $K$ is imaginary quadratic, the well-known
Kronecker limit formula expresses $\gamma_K$ in terms of special values of the Dedekind $\eta$-function. 

An alternative formula for $\gamma_K$ is given by
\begin{equation}
\label{define}
\gamma_K=\lim_{s \downarrow 1}\Big(\frac{\zeta'_K(s)}{\zeta_K(s)}+\frac{1}{s-1}\Big),
\end{equation}
which shows that $\gamma_K$ is the constant part in the Laurent series of the logarithmic derivative
of $\zeta_K(s)$. Using the Hadamard factorization \eqref{hadimassa} one can relate $\gamma_K$ to the
sum of the reciprocal zeros of $\zeta_{K}(s)$, cf. 
\cite[p. 1452]{FLM}. 
Indeed, in a lot of the literature
the logarithmic derivative of the right hand side of \eqref{hadimassa} is the 
starting point in studying $\gamma_K$.
The main tool of Ihara, cf. \cite[p. 411]{I}, is an ``explicit" formula for the prime
function
$$\Phi_K(x)=\frac{1}{x-1}\sum_{N\mathfrak p^k\le x}\Big(\frac{x}{N\mathfrak p^k}-1\Big)
\log N\mathfrak p,~~x>1,$$
relating it to the zeros of $\zeta_K(s)$.
\par Given any Dirichlet series $L(s)$ with
a pole at $s=1$, we can define its Euler-Kronecker
constant as the constant part in the Laurent series of its logarithmic derivative (if this
constant exists). In Moree \cite{MIndia} this is considered in case when $S$ is a multiplicative
set of integers (that is, for coprime integers $m$ 
and $n$ one has $mn\in S$ if and only if both $m$ and $n$ are in $S$) and
$L_S(s)=\sum_{n\in S}n^{-s}$ is its associated Dirichlet series.

Another alternative formula for $\gamma_K$ is given by
\begin{equation}
\label{EKlog}
\gamma_K=\lim_{x\rightarrow \infty}\Big(\log x-\sum_{N \mathfrak p\le x}\frac{\log N \mathfrak p}
{N\mathfrak p-1}\Big).
\end{equation}
This result is due to de la Vall\'ee-Poussin (1896) 
in case $K=\mathbb Q$ and can be easily generalized to 
other number fields and settings, cf. \cite{Hash,HIKW}.

Ihara \cite[Theorem 1 and Proposition 3]{I} proved that GRH (Conjecture \ref{GRH} below) implies that there are absolute constants
$c_1,c_2>0$ such that
\begin{equation}
\label{iha}
-c_1\log |d_K| < \gamma_K < c_2\log \log |d_K|,
\end{equation}
where $d_K$ 
denotes the discriminant $K/\mathbb Q$. 
Tsfasman \cite{T} showed that
the above lower bound is sharp, namely, assuming GRH he proved that
$$\lim \inf \frac{\gamma_K}{\log |d_K|}\ge -0.13024\ldots ,$$
where we range over the number fields $K$ with $|d_K|\rightarrow \infty$.
Later Badzyan \cite{Badz} proved that one can take $c_1=(1-1/\sqrt{5})/2\approx 0.276393$.
It is an open problem whether this is sharp.

\subsection{The Euler-Kronecker constant for cyclotomic fields}
It is a natural question to ask how the Euler-Kronecker constant
varies over families of number fields such as
quadratic fields and (maximal) cyclotomic fields. After quadratic fields, cyclotomic
fields have been most intensively studied (see \cite{Lang,W} for book length 
treatments). Many of the associated quantities of a cyclotomic field
$\mathbb Q(\zeta_m)$ are explicitly
known. Relevant examples for us are their ring of integers, $\mathbb Z[\zeta_m]$, 
and their discriminant
\begin{equation}
\label{generaldiscri}
d_{\mathbb Q(\zeta_n)}=(-1)^{\varphi(n)/2} n^{\varphi(n)} \prod_{p|n}p^{-\frac{\varphi(n)}{(p-1)}}.
\end{equation}
Moreover, the splitting of a rational prime $p$ 
into prime ideals in $\mathbb Z[\zeta_m]$
of a cyclotomic field follows an easy pattern, see e.g. \cite[Theorem 4.16]{N}, which
we recall here.
\par For $p$ a prime not diving the integer $m$, we define ${\rm ord}_p(m)$ to be the (multiplicative) 
order of $p$ in $(\mathbb Z/m\mathbb Z)^*$
\begin{lem}[Cyclotomic reciprocity law] 
\label{washington} Let $K=\mathbb Q(\zeta_m)$. 
If the prime $p$ does not divide
$m$ and $f={\rm ord}_p(m)$, then the principal 
ideal $p{\mathcal O}_{K}$ factorizes as $\mathfrak{p}_1\cdots \mathfrak{p}_g$ with
$g=\varphi(m)/f$ and all $\mathfrak{p_i}$ are distinct and of degree $f$.

However, if $p$ divides $m$, $m=p^am_1$ with $p\nmid m_1$ and $f={\rm ord}_{p}(m_1)$, 
then $p{\mathcal O}_{K}=(\mathfrak{p}_1\cdots \mathfrak{p}_g)^e$ with $e=\varphi(p^a)$, 
$g=\varphi(m_1)/f$ and all $\mathfrak{p_i}$ are distinct and of degree $f$.
\end{lem}

For notational convenience we will write $\gamma_m$ instead of 
${\gamma}_{\mathbb Q(\zeta_m)}$.
Our main focus is on the case where $m=q$ is a prime 
(unless specified otherwise, $m$ denotes a positive integer and
$p$ and $q$ primes).
Then we have
\begin{equation}
\label{oud}
\zeta_{\mathbb Q(\zeta_q)}(s)=\zeta(s)\prod_{\chi\ne \chi_0}L(s,\chi),
\end{equation}
where $\chi$ ranges over the non-trivial characters modulo $q$,
leading to
\begin{equation}
\label{gammaq}
\gamma_q=\gamma+\sum_{\chi\ne \chi_0}\frac{L'(1,\chi)}{L(1,\chi)}.
\end{equation}
Thus the behaviour of $\gamma_q$ is related to that of $L(s,\chi)$ and $L'(s,\chi)$
at $s=1$. (Here and in the rest of the paper we often use the fundamental
fact that $L(1,\chi)\ne 0$.)
\par Let $\mathbb Q(\zeta_m)^+$ denote the maximal real subfield of 
$\mathbb Q(\zeta_m)$ and $\gamma_m^+$ its Euler-Kronecker 
constant. In that case we find
\begin{equation}
\label{maxreal}
\zeta_{\mathbb Q(\zeta_q)^+}(s)=\zeta(s)\prod_{\substack{\chi\ne \chi_0\\\chi(-1)=1}} L(s,\chi). 
\end{equation}
Logarithmic differentiation of the latter product identity then
yields
\begin{equation}
\label{EKL}
\gamma_q^+ = \gamma + \sum_{\substack{\chi\ne \chi_0\\\chi(-1)=1}}\frac{L'(1,\chi)}{L(1,\chi)}.
\end{equation}
\subsection{Ihara's conjecture}
Ihara made a conjecture on $\gamma_m$ 
based on numerical observations for $m\le 8000$,
which we here only formulate in case $m=q$ is prime.
\begin{conj}[Ihara's conjecture \cite{I2}] Let $q\ge 3$ be a prime.\hfil\break
{\textup{1)}} $\gamma_q>0$ (`very likely');\\
{\textup{2)}} For fixed $\epsilon>0$ and $q$ sufficiently large we have
$$\frac{1}{2}-\epsilon \le \frac{\gamma_q}{\log q}\le \frac{3}{2}+\epsilon.$$
\end{conj}
The most extensive computations on $\gamma_q$ to date were carried out
by Ford et al.~\cite{FLM}.\\
\centerline{\includegraphics[height=8cm, angle=270]{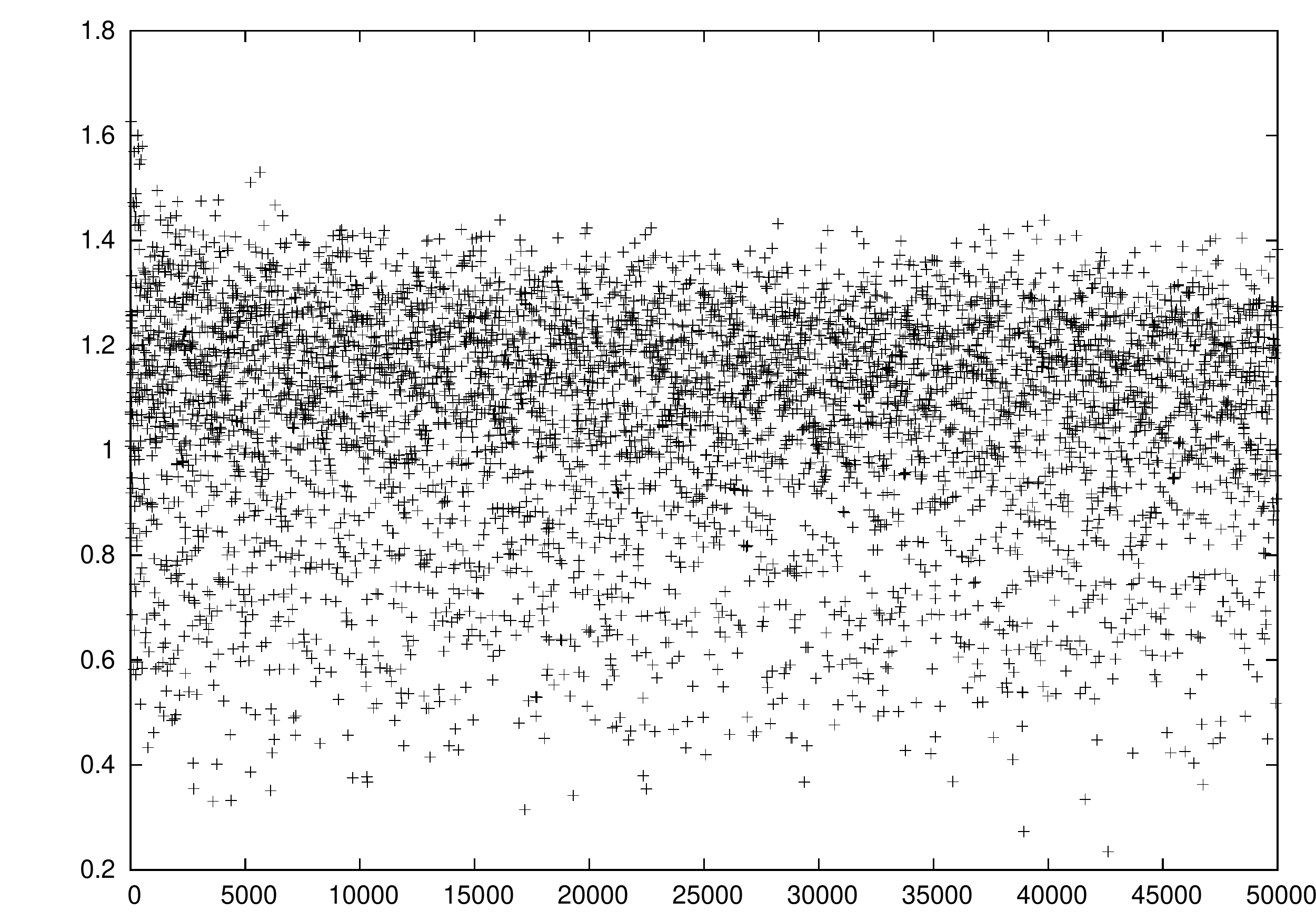}}
\centerline{Scatterplot of $\gamma_q/\log q$ for $q\le 50000$}\\
\\
\noindent The largest value of $\gamma_q/\log q$ among $q\le 30000$ 
equals $1.626\ldots$ and occurs at $q=19$. The smallest is
$0.315\ldots$ and occurs at $q=17183.$
It is a consequence of \eqref{iha}, \eqref{generaldiscri} 
and Badzyan's result mentioned above that, under
GRH, there exists a constant $c_2'>0$ such that
\begin{equation}
\label{weakbounds}
-(1-1/\sqrt{5})q(\log q)/2<\gamma_q<c_2'\log q.
\end{equation}
Ihara \cite{I} showed that $\gamma_q\le (2+o(1))\log q$ assuming 
ERH (Conjecture \ref{conERH} below). The lower bound
in \eqref{weakbounds} turns out to be very weak. Ihara et al. \cite{IKMS} proved that for any $\epsilon>0$ one has
$|\gamma_q|=O_{\epsilon}(q^{\epsilon})$ and, under GRH,
$|\gamma_q|=O(\log^2 q).$ We will see in Section \ref{resconj} that 
these bounds can be sharpened considerably.
\subsection{Kummer's conjecture}
Let $h_1(q)$ be the ratio of the class number $h(q)$ of 
$\mathbb Q(\zeta_q)$ and the class number $h_2(q)$ of its maximal 
real subfield $\mathbb Q(\zeta_q+\zeta_q^{-1})$, that is,
$h_1(q)=h(q)/h_2(q)$. Kummer proved that this is an integer. 
It is now called the first factor of the class number of $h(q)$. 
\par In 1851 Kummer \cite{Kummer} published a review of the main results that he and others had
discovered about cyclotomic fields. In this elegant report he 
made the following conjecture.
\begin{conj}[Kummer's conjecture \cite{Kummer}]
Put $$G(q)=\Big(\frac{q}{ 4\pi^2}\Big)^\frac{q-1}{4}\text{~and~~~}\;r(q)=\frac{h_1(q)}{G(q)}.$$
Then asymptotically $r(q)$ tends to 1.
\end{conj}
In fact he claimed to have a proof that he would publish later
together with further developments (but never did). 
Kummer himself laboriously computed $r(q)$ for $q<100$. 
This was extended over time by many authors, more recently
by Shokrollahi \cite{shokro}.
He showed that the largest value of $r(q)$ among $q\le 10000$ 
equals $1.556562\ldots$ and occurs at $q=5231$. The smallest is
$0.642429\ldots$ and occurs at $q=3331$.
\par In 1949 
Ankeny and Chowla \cite{AC1,AC2}
made some progress by showing that
\begin{equation}
\label{ACwegrestaurant}
\log r(q)=o(\log q).
\end{equation}
Siegel \cite{Siegel}, who was
unaware of the earlier work
of Ankeny and Chowla, proved a weaker version of \eqref{ACwegrestaurant} and 
was one of the first to cast doubt on the truth of Kummer's conjecture.
From \eqref{ACwegrestaurant} we infer that
$$\log h_1(q) \sim \frac{q}{4}\log q,$$
and thus that there are only finitely many primes $q$ such
that $\mathbb Q(\zeta_q)$ has class number one.
This was made effective by Masley and Montgomery \cite{mamo}, 
who showed that $|\log r(q)|<7\log q$ for $q>200$, which is
strong enough to establish Kummer's conjecture that 
$h_1(q)=1$ if and only if $q\le 19$. 
This result is their
key ingredient in determining all cyclotomic fields
having class number 1. 
In proving their upper bound, Masley and Montgomery used zero-free
regions of $L$-functions.
This idea was refined by Puchta \cite{Puchta}, with a further improvement
by Debaene \cite{Debaene}, to obtain an upper bound 
for $\log r(q)$ that depends on a Siegel zero, if it exists.
\par Not
surprisingly $h_1$ is eventually monotonic; however, no
beginning prime is yet known. In this direction 
Fung et al. \cite[Theorem 1]{FGW} showed that if $E$ is an elliptic
curve over $\mathbb Q$ for which the associated $L$-function
has a zero of order at least 6 in $s=1$, then it is possible
to find an explicit prime $q_0$ for which $h_1(q_2)>h_1(q_1)$, 
whenever $q_2>q_1\ge q_0$. It is believed that one can take $q_0=19$.
\par Just like $\gamma_q$ 
in \eqref{gammaq}, $h_1(q)$
is also related to special values of Dirichlet $L$-series. 
Hasse \cite{Hasse} showed that
\begin{equation}
\label{hasse}
r(q)=\frac{h_1(q)}{G(q)}=\prod_{\chi(-1)=-1}L(1,\chi),    
\end{equation}
where the product is over all the odd characters modulo $q$.
It follows from this, \eqref{oud} and \eqref{maxreal} that
\begin{equation*}
r(q)=\lim_{s \downarrow 1}\frac{\zeta_{\mathbb Q(\zeta_q)}(s)}
{\zeta_{\mathbb Q(\zeta_q)^+}(s)}.
\end{equation*}
Indeed, using the definition of Euler-Kronecker constant
we find the Taylor series expansion around $s=1$
\begin{equation}
\label{Taylor}
\frac{\zeta_{\mathbb Q(\zeta_q)}(s)}
{\zeta_{\mathbb Q(\zeta_q)^+}(s)}=
r(q)(1+(\gamma_q-\gamma_q^+)(s-1)+O_q((s-1)^2)),
\end{equation}
involving both of the main actors of this 
survey\footnote{I have not come across this formula in the literature.}.
\subsection{Similarities between the two conjectures}
The remaining part of this survey will make clear that the 
quantities
\begin{equation}
\label{similar}
\frac{\gamma_q}{\log q}\quad\quad \text{and}\quad\quad 1-2|\log r(q)|
\end{equation}
have very similar analytical 
properties. 
Indeed, this analogy implies that the Euler-Kronecker analogue of the Kummer
conjecture is that asymptotically
$$\gamma_q\sim \log q.$$
The numerical computations mentioned 
above suggest that both 
quantities in \eqref{similar} are bounded, whereas if
one believes in some standard conjectures in analytic number theory (delineated
in the next section), they can be sporadically
very negative. Various researchers in this area 
believe that it is the $\log \log \log$ devil that ruins both the Kummer and Ihara conjecture (see 
Section \ref{unleashed}).
\section{Preliminaries}
\subsection{Standard conjectures used}
\label{standard}
The results we are going to present depend on
some standard conjectures on the prime distribution
that we briefly recall in this section.
\par Let $\mathcal A=\{a_1,\ldots,a_s\}$ be a set 
consisting of $s$ distinct natural numbers.
We define
$$m(\mathcal A)=\sum_{k=1}^s\frac{1}{a_i}.$$
The set $\mathcal A$ is said to be admissible if there does
not exist a prime $p$ such that $p|n\prod_{i=1}^s(a_in+1)$ for
every $n\ge 1$.
Note that if there is such a prime factor $p$, then $p\le s+1$.
The sequence $\{a(i)\}_{i=1}^{\infty}=\{2,6,8,12,18,20,26,30,32,\ldots\}$ 
has the property that any finite sub sequence is an admissible set.
It is called ``the greedy sequence of prime offsets" and is 
sequence A135311 in the Online Encyclopedia of Integer Sequences (OEIS).
\begin{conj}[Hardy-Littlewood]
Suppose $\mathcal A=\{a_1,\ldots,a_s\}$ is an admissible set.
Then the number of primes $n\le x$ such that the integers 
$a_1n+1,\ldots,a_sn+1$ are all prime is $\gg_{\mathcal A} x\log^{-s-1}x$.
\end{conj}
Actually, the full Hardy-Littlewood conjecture gives an asymptotic, 
rather than a lower bound. It is this full version that was used by
Croot and Granville \cite{CG} to study how many 
primes $q\le x$ satisfy $r(q)=\alpha+o(1)$, with
$\alpha>0$ and fixed.
\par As usual, we let $\pi(x;d,a)$ denote the number of primes $p\le x$ 
satisfying $p\equiv a\pmod*{d}$, $\pi(x)$ 
the prime counting function, Li$(x)$ the logarithmic integral and 
$\mu$ the M\"obius function.
\begin{conj}[Elliott-Halberstam]
For any $\epsilon>0$ and $C>0$ we have
$$\sum_{k<x^{1-\epsilon}}\max_{(l,k)=1}\max_{y\le x}
\left|\pi(y;k,l)-\frac{{\rm Li}(x)}{\varphi(k)}\right|\ll_{\epsilon,C}\frac{x}{(\log x)^C}.$$
\end{conj}
\begin{conj}[Extended Riemann Hypothesis]
\label{conERH}
Every Dirichlet series $L(s,\chi)$ satisfies the Riemann Hypothesis.
\end{conj}
\begin{conj}[Generalized Riemann Hypothesis]
\label{GRH}
Every Dedekind zeta function $\zeta_K(s)$ satisfies the Riemann Hypothesis.
\end{conj}
In places where one uses GRH for a general number field, for a cyclotomic
number field ERH suffices, as their Dedekind zeta function decomposes as a product
of Dirichlet $L$-series, cf. \eqref{oud}.
\par For most results quoted below a weaker form of these conjectures suffices.
For reasons of brevity we leave out the details and refer the reader to 
the original publications. Also for brevity we will refer to the above conjectures by the abbreviations HL, EH, ERH and GRH, respectively.
\subsection{The distribution of $m(\mathcal A)$}
Crucial for obtaining results on both 
the Kummer and the Ihara conjecture is an understanding of
the distribution of $m(\mathcal A)$ as $\mathcal A$ ranges over the
admissible sets. We put $\mathcal M=\{m(\mathcal A):{\mathcal A}~{\rm~is~admissible}\}$ 
and let $\overline{\mathcal{M}}$ be the closure of $\mathcal M,$ that is, the set
of limit points of sequences of elements of $\mathcal M$ that do converge.
\par Granville showed that the following 1988 conjecture by Erd\H{o}s 
is true\footnote{The authors of \cite{FLM}, unaware of Granville's work and the fact that they were dealing with
a(n) (ex-)conjecture of Erd\H{o}s, gave a short different proof 
using a 1961 paper of... Erd\H{o}s \cite{Erd} himself!}. 
\begin{Thm}[Granville \cite{Gr}]
\label{unbounded}
There is a sequence of admissible sets $\mathcal A_1,\mathcal A_2,\ldots$ 
such that $\lim_{j\rightarrow \infty}m(\mathcal A_j)=\infty$. 
\end{Thm}
\begin{cor}
\label{dicht}
We have $\overline{\mathcal M}=[0,\infty]$.
\end{cor}
\begin{proof}
Given any $x>0$ and $\delta>0$, there is an admissible set $\mathcal A$ 
 with $m(\mathcal A)>x$ consisting of integers
all $>1/\delta$. As any subset of an admissible set is also 
admissible, there is a subset ${\mathcal A}'$ of $\mathcal A$ with
$|m(\mathcal A')-x|<\delta$.
\end{proof}
Another issue is that of finding admissible subsets $\mathcal A\subseteq [1,x]$ having
large $m(\mathcal A)$. In this direction Granville proved the
following result.
\begin{prop}[Granville \cite{Gr}]\label{admissiblecounting} $~$\\
\textup{1)} For any sufficiently large $x$ there is an admissible set $\mathcal A$, 
which is a subset of $[1,x]$, with $m(\mathcal A)\ge (1+o(1))\log \log x$.\\
\textup{2)} There exists a constant $c>0$ such that if $\mathcal A$ is an
admissible subset of $[1,x]$, then $m(\mathcal A)\le c\log \log x$.
\end{prop}
Granville believes one can take $c=1+\epsilon$, for any $\epsilon>0$, provided
that $x$ is sufficiently large.
If true, this would imply that part 1 is best possible.

\section{The constants $\gamma_q$: results and conjectures}
\label{resconj}
On applying \eqref{EKlog}  and Lemma \ref{washington} we obtain
\begin{equation}
\label{watowat}
\gamma_q=-\frac{\log q}{q-1}-S(q)-
\lim_{x\rightarrow \infty}\Big(\log x -(q-1)\sum_{\bfrac{p\le x}{p\equiv 1\pmod*q}}\frac{\log p}{p-1}\Big),
\end{equation}
where 
$$S(q)=(q-1)\sum_{\bfrac{p \neq q}{\text{ord}_p(q)\ge 2}}\frac{\log p}{p^{\text{ord}_p(q)}-1}.$$
By Lemma \ref{washington} the only rational primes splitting into prime ideals of prime
norm are $q$ and all the primes $p\equiv 1\pmod*q$.
They are responsible for the first, respectively third term on the
right hand side of \eqref{watowat}. The term $S(q)$ is the contribution of the prime ideals
lying above the remaining rational primes. Using estimates for linear
forms in logarithms, it can be shown that $S(q)\le 45$ and even that
for any fixed $\epsilon>0$ we have $S(q)<\epsilon$ for
$(1+o(1))\pi(x)$ primes $q\le x$ \cite[Theorem 3]{FLM}. Since, 
as we will see, $\gamma_q$ has normal order $\log q,$ it follows
that the first two terms in \eqref{watowat} are \emph{error terms}.
\par The idea now is to approximate $\gamma_q$ by choosing
a suitable value for $x$ in \eqref{watowat}. 
In principle one wants to have $x$ small, but the irregularities in the
distribution of the primes do not allow us to take $x$ too small. 
The Bombieri-Vinogradov theorem allows us to take 
$x=q^{2+\delta}$ for any $\delta>0$ with the possible exception of a 
thin set of primes. Using the Brun-Titchmarsh inequality one can
bring this down to $x=q^2$. Likewise, assuming EH one can go down to $x=q^{1+\delta}$. This approach leads
to the following result. 
\begin{lem}[Ford et al.~\cite{FLM}] 
\label{vijf}
Given $r>1$ write
\begin{equation}
\label{primesum}
E_r(q)=\gamma_q-r\log q +q\sum_{\bfrac{p\le q^r}{p\equiv 1\pmod*q}}\frac{\log p}{p-1}.
\end{equation}
\textup{1)} For all $C>0$ we have
$E_2(q)=O_C(\log \log q)$,
with at most $O\big(\frac{\pi(x)}{(\log x)^{C}}\big)$ exceptions $q\le x$.\\
\textup{2)} Assuming EH, we have for fixed
$\epsilon>0$ and $C>0$ that $E_{1+\epsilon}(q)=O_{C,\epsilon}(\log \log q)$, with
at most  $O\big(\frac{\pi(x)}{(\log x)^{C}}\big)$ exceptions $q\le x$.\\
\textup{3)} Assuming ERH, we have $E_2(q)=O(\log \log q)$.
\end{lem}
Before we consider how the large the prime sum in \eqref{primesum} 
with $0<r\le 2$ can be,
we remark that it is usually small.
\begin{prop}[Ford et al.~\cite{FLM}]
\label{geennaam}
Uniformly for $z\ge 2$, $\delta>0$ and
$0<\epsilon\le 1$, the number of primes $q\le x$ for which
$$q\sum_{\bfrac{p\le q^{1+\epsilon}}{p\equiv 1\pmod*q}}\frac{\log p}{p-1}\ge 
\delta{~\log q}$$
is $O(\epsilon \pi(x)/\delta)$.
\end{prop}
How small $\gamma_q$ can be is determined by how large the prime sum in \eqref{primesum} can be.
\begin{prop}
There exists an absolute constant $c>0$ such that on a set of primes
of natural density 1 we have 
$$-c\log \log q<\frac{\gamma_q}{\log q}<(2+\epsilon)\log q,$$
with $\epsilon>0$ arbitrary and fixed.
\par Under ERH these estimates hold for all primes $q$ large enough.
\end{prop}
\begin{proof}
On writing the primes $p\equiv 1\pmod*q$ that satisfy $p\le q^2$ as
$a_1q+1,\ldots,a_sq+1,$ and noting that 
${\mathcal A}:=\{a_1,\ldots,a_s\}$ is an admissible set, we obtain by
Proposition \ref{admissiblecounting} that
$$q\sum_{\bfrac{p\le q^2}{p\equiv 1\pmod*q}}\frac{\log p}{p-1}
< 2m({\mathcal A})\log q\ll \log q\log \log q.$$
Now the unconditional statement
is obtained on invoking part 1 of Lemma \ref{vijf}. 
\par Under ERH the upper bound is due to Ihara \cite{I}
and the lower bound
to Badzyan\footnote{He assumes GRH. The reproof given
in \cite[p. 1470]{FLM} shows that ERH is sufficient.} \cite{Badz}.
\end{proof}
 In the next section we will see that for $r=2$ the prime sum in \eqref{primesum}
can be quite large if we assume HL.
\subsection{Assuming HL}
Armed with HL and Lemma \ref{vijf}, it is easy to give a 
conditional disproof of part 1 of Ihara's conjecture.
\begin{Thm}
\label{maaah}
Suppose that HL is true
and that $\mathcal A$ is an admissible set. Then 
one has 
$$\gamma_q<(2-m(\mathcal A)+o(1))\log q$$ for $\gg x\log^{-\#{\mathcal A}-1}x$ primes
$q\le x$.
\end{Thm}
\begin{proof}
Let $a_1,\ldots,a_s$ be the elements of  $\mathcal A$.
By HL 
there are infinitely many primes $q$ such that 
infinitely often $a_1q+1,\ldots,a_sq+1$ are all prime and in addition
$a_sq+1\le q^2$. Then
$$q\sum_{\bfrac{p\le q^2}{p\equiv 1\pmod*q}}\frac{\log p}{p-1}>q\sum_{i=1}^s\frac{\log q}{a_iq}=m(\mathcal A)\log q.$$
The proof now easily follows from the part 1 of Lemma \ref{vijf} with any $C>s$.
\end{proof}
A computer calculation gives that
$\mathcal A=\{a(1),\ldots,a(2088)\}$ satisfies $m(\mathcal A)>2$, 
where $a(1),a(2),\ldots$ is the sequence of integers introduced in Section \ref{standard}.
We thus obtain the following corollary of Theorem \ref{maaah}.
\begin{cor}
Assume HL.
Then part 1 of Ihara's conjecture is false for infinitely many primes $q$. 
\end{cor}
Unconditionally we only have the following result.
\begin{Thm}[Ford et al.~\cite{FLM}]
We have $\gamma_{964477901}= -0.1823\ldots,$ and so part 1 of Ihara's conjecture is
false for at least one prime $q$.
\end{Thm}
This looks perhaps easy, but was made possible only by a new, fast algorithm
developed by the authors of \cite{FLM} (it requires computation of $L(1,\chi)$ for
all characters modulo $q$). The prime $q=964477901$ has the property that
$aq+1$ is prime for $a\in \{2,6,8,12,18,20,26,30,36,56,\ldots\}$. 
It is easy to approximate the above value of $\gamma_q$ by taking a large 
$x$ in formula \eqref{watowat}. The authors of \cite{FLM} believe that
if there is a further $q$ with $\gamma_q<0$, then its computation will be
hopelessly infeasible.
\par Since by Theorem 
\ref{unbounded} one can find admissible $\mathcal A$ with $m(\mathcal A)$ arbitrarily
large, we obtain the following result from Theorem \ref{maaah}.
\begin{Thm}[Ford et al.~\cite{FLM}]
Assume HL. Then
$$\lim \inf_{q\rightarrow \infty}\frac{\gamma_q}{\log q}=-\infty.$$
\end{Thm}
Thus, conditionally, $\gamma_q$ can be very negative. 
This happens not frequently since Mourtada and Kumar Murty \cite{MKM} 
showed unconditionally that the set of 
primes $q\le x$ such that $\gamma_q\le -11\log q$ is
of size $o(\pi(x))$.
In Section \ref{unleashed} we speculate how negative $\gamma_q$ 
as a  function of $q$ can be. 
\subsection{Assuming EH (and HL)}
\label{gr}
The prime sum in \eqref{primesum} cannot be too small by Proposition \ref{geennaam},
and on invoking the part 2 of Lemma \ref{vijf} we obtain the following result.
\begin{Thm}[Ford et al.~\cite{FLM}]
Assume EH. Let $\epsilon>0$ be arbitrary.
For a density 1 sequence of primes $q$ we have
$$1-\epsilon<\frac{\gamma_q}{\log q}<1+\epsilon.$$
\end{Thm}
This describes the situation for the bulk of the primes. However, if one assumes 
in addition HL, one can say something about the irregular behaviour.
\begin{Thm}
\label{maaah2}
Suppose that both EH and HL are true.
If $\mathcal A$ is an admissible set, then
one has 
$$\gamma_q=(1-m(\mathcal A)+o(1))\log q$$
for  $\gg_{\mathcal A} x\log^{-\#{\mathcal A}-1}x$ primes
$q\le x$.
\end{Thm}
\begin{proof}[Sketch of proof] By reasoning as in the proof 
of Theorem \ref{maaah}, we obtain 
$\gamma_q\le (1-m(\mathcal A)+o(1))\log q$.
The reverse inequality is obtained on using sieve
methods to find enough primes $q\le x$ with
$qa+1$ prime for $a\in \mathcal A$ and not
prime for $a\not\in \mathcal A$ and $a\le q^{\epsilon}$,
see \cite[p. 1465]{FLM} for details.
\end{proof}
Now using that $\overline{\mathcal M}=[0,\infty]$ (Corollary \ref{dicht}), 
we obtain the following result.
\begin{Thm}[Ford et al.~\cite{FLM}]
Assume EH and HL.  Then the set  
$$\Theta:=\left\lbrace\frac{\gamma_q}{\log q}:q\text{~prime}\right\rbrace$$ is 
dense in $(-\infty,1]$.
\end{Thm}
We propose the following conjecture.
\begin{conj}
\label{zucht}
Let ${\mathcal A}$  be any admissible set. 
If EH and HL are both true, then
$1-m(\mathcal A)$ is a limit point of the set $\Theta$.
\end{conj}

\subsection{Cyclotomic Euler-Kronecker constants on average}
\label{onaverage}
Kumar Murty \cite{KM} proved unconditionally that
$$\sum_{Q/2<q\le Q}|\gamma_q|\ll (\pi(Q)-\pi(Q/2))\log Q.$$ 
Fouvry
\cite {Fouvry} showed that uniformly for
$M\ge 3$ one has the equality
$$\frac{1}{M}\sum_{M/2<m\le M}|\gamma_m|=\log M+O(\log \log M),$$
if one ranges over the integers $m$, rather than the primes $q.$

\section{The Kummer conjecture: (conditional) results}
The orthogonality property of characters gives us
$$\sum_{\chi(-1)=-1}\log L(s,\chi)=
\frac{q-1}{2}\sum_{p^m\equiv  \pm 1\pmod*q}\pm \frac{1}{mp^{ms}},$$
where the latter notation is shorthand for
$$\sum_{p^m\equiv  1\pmod*q}\frac{1}{mp^{ms}} 
-\sum_{p^m\equiv  -1\pmod*q}\frac{1}{mp^{ms}}.$$
From Hasse's formula \eqref{hasse} we have that
\begin{equation}
\label{loggie}
\log r(q)=\frac{q-1}{2}\lim_{x\rightarrow \infty}\Big(\sum_{m\ge 1}\frac{1}{m}
\sum_{\bfrac{p^m\le x}{p^m\equiv \pm 1\pmod*q}}\pm \frac{1}{p^m}\Big).
\end{equation}
We denote the limit by $f_q$. 
Note that Kummer's conjecture is equivalent with $f_q=o(1/q)$.
Formula \eqref{loggie} should be compared to formula 
\eqref{watowat}. As in that formula, one tries to choose $x$ as small 
as possible so that the resulting error is still reasonable. In doing so, also
here Bombieri-Vinogradov theorem and Brun-Titchmarsh inequality come into play.
The main contribution to $f_q$ comes from the
term with $m=1$. Taking all this into account,
Granville \cite{Gr} showed that if Kummer's conjecture is true, then for every $\delta>0$ we must have
$$\sum_{\bfrac{p\le q^{1+\delta}}{p\equiv \pm 1\pmod*q}}\pm \frac{1}{p}=o\Big(\frac{1}{q
}\Big),$$ for all but at most $\ll x/\log^3 x$ primes $q\le x$.
\par Using this approach Granville showed that
$$1/c\le r(q)\le c$$
for a positive proportion $\rho(c)$ of primes $p\le x$, where $\rho(c)\rightarrow 1$ as $c\rightarrow \infty$. Ram Murty and Yiannis Petridis \cite{MP} improved this as follows.
\begin{Thm}
There exists a positive constant $c$ such that for a sequence of primes with natural density 
$1$ we have $$c^{-1}\le r(q)\le c.$$ If EH is true, then
we can take $c=1+\epsilon$ for any fixed $\epsilon>0$.
\end{Thm}
Thus Ram Murty and Yiannis Petridis showed that a weaker version  of Kummer's conjecture holds true.
Yet, if both EH and HL are true, Kummer's conjecture  itself is false and, moreover, we
have the following much stronger result.
\begin{Thm} [Granville \cite{Gr}]
Put 
$$\Omega=\{r(q):q\text{~is~prime}\}.$$
Assume both HL and EH. Then the sequence 
$\Omega$ has $[0,\infty]$ as set of 
limit points.
\end{Thm}
This result follows from Corollary \ref{dicht} and the following.
\begin{Thm}
\label{zucht1}
If EH and HL are both true, then, for any admissible set 
${\mathcal A}$, the numbers
$e^{m(\mathcal A)/2}$ and $e^{-m(\mathcal A)/2}$ are both
limit points of $\Omega$.
\end{Thm}

\section{The log log log devil unleashed}
\label{unleashed}
Regarding the extremal behaviour of $r(q)$ and $\gamma_q/\log q$, we 
enter the realm of speculation, 
following Granville \cite[Section 9]{Gr}.
\begin{Speculation}[Granville \cite{Gr}]
For all primes $q$, we have
\begin{equation}
\label{Granvillebound}
(-1+o(1))\log \log \log q\le 2\log r(q)\le (1+o(1))\log \log \log q.
\end{equation}
These bounds are best possible in the sense that there exist two infinite 
sequences of primes for which the lower, respectively upper bound are attained.
\end{Speculation}
The same line of thought for $\gamma_q$ gives rise to
the following speculation.
\begin{Speculation}
For all primes $q$, we have
\begin{equation}
\label{FLMbound}
\frac{\gamma_q}{\log q}\ge (-1+o(1))\log \log \log q.
\end{equation}
The bound is best possible in the sense that there exists an infinite 
sequence of primes for which the bound is attained.
\end{Speculation}
We will now sketch the motivation for these two
speculations and do this in parallel, to bring
out the analogy in the reasoning more clearly. 
The speculations require the assumption that
primes are both more regularly and more irregularly
distributed than can be currently established.
\par For convenience let us 
write $L_2=\log \log q$ and $L_3=\log \log \log q$.
We assume that there exists an absolute constant $A>0$ 
for which
we can take $x=q(\log q)^A$ in \eqref{watowat}, such 
that the estimate 
\begin{equation}
\label{qlog3}
\gamma_q=\log q-
q\sum_{\bfrac{p\le q(\log q)^A}{p\equiv  1\pmod*q}}\frac{\log p}{p-1}+E(q),
\end{equation}
with $E(q)=o(L_3\log q)$ holds true.
Now note that
\begin{equation}
\label{ikzatzijn}
\sum_{\bfrac{p\le q(\log q)^A}{p\equiv  1\pmod*q}}\frac{\log p}{p-1}
=\sum_{\bfrac{p\le q(\log q)^A}{p\equiv 1\pmod*q}}\frac{\log q}{p}
\Big(1+O_A\Big(\frac{L_2}{\log q}\Big)\Big),
\end{equation}
where we used that the $\log p$ appearing 
on the left hand side of \eqref{ikzatzijn} satisfies
$\log p=\log q+O_A(L_2)$. Combining \eqref{ikzatzijn} and \eqref{qlog3} then yields
\begin{equation}
\label{gammaratio}
\frac{\gamma_q}{\log q}=
1-\sum_{\bfrac{2q-1\le p\le q(\log q)^3}{p\equiv 1\pmod*q}}\frac{q}{p}
\Big(1+O_A\Big(\frac{L_2}{\log q}\Big)\Big)+\dfrac{E(q)}{\log q}.
\end{equation}
Granville \cite[p. 335]{Gr} 
makes some speculations about the distribution of 
prime numbers
that would ensure that one can go down
to $x=q(\log q)^3$ in the
Kummer problem and lead to\footnote{
Having the larger error term $o(L_3)$ would also 
suffice for
our purposes.}
\begin{equation}
\label{Kummerratio}
\log r(q)=\frac{q-1}{2}
\sum_{\bfrac{p\le q(\log q)^3}{p\equiv  \pm 1\pmod*q}}\pm \frac{1}{p}+O\Big(\frac{1}{\sqrt{\log q}}\Big).
\end{equation}
By the Brun-Titchmarsh theorem there exists a constant $c>0$ such
that for all $x\ge 2q-1$ we have
$$\max\{\pi(x;q,-1),\pi(x;q,1)\}\le c\frac{x}{(q-1)\log(x/q)}.$$
Using this it is easy to deduce that
\begin{equation}
\label{geenzinmeer}
\sum_{\bfrac{2q+1\le p\le q(\log q)^A}{p\equiv  1\pmod*q}}\frac{1}{p}\le \frac{c}{q}(L_3+O_A(1)).
\end{equation}
Combining this estimate with \eqref{gammaratio} gives
\begin{equation}
\label{gammaratio2}
\frac{\gamma_q}{\log q}\ge 
1-cL_3+O_A(1)+\dfrac{E(q)}{\log q}.
\end{equation}
Similarly, combining \eqref{geenzinmeer} 
with $A=3$ and \eqref{Kummerratio} yields
\begin{equation}
\label{Kummerratio2}
\log r(q)\le cL_3/2+O(1).
\end{equation}
It follows from \eqref{Kummerratio} that 
\begin{equation*}
\log r(q)\ge -\frac{q-1}{2}
\sum_{\bfrac{2q-1\le p\le q(\log q)^3}{p\equiv  -1\pmod*q}} \frac{1}{p}+O\Big(\frac{1}{\sqrt{\log q}}\Big),
\end{equation*}
and a similar argument as before now yields
\begin{equation}
\label{Kummerratio3}
\log r(q)\ge -cL_3/2+O(1).
\end{equation}
Montgomery and Vaughan \cite{MV} have shown that we may take $c=2$ and it is conjectured that
one may take $c=1+o(1)$. 
If this is so, then combining \eqref{Kummerratio2}
with \eqref{Kummerratio3} yields
\eqref{Granvillebound}.
Likewise, \eqref{gammaratio2} gives rise to the
lower bound \eqref{FLMbound}.

\par The final step is to argue why the bounds
\eqref{Granvillebound} and \eqref{FLMbound} are
best possible. 
We will only do so for the easier case of the bound
\eqref{FLMbound}.
We assume the
Hardy-Littlewood conjecture in a stronger form, namely in 
its original asymptotic
form. Then one can argue that for any admissible set
$\mathcal A$ with elements $\le z$, we only find
enough primes $q$ for which $p=qa+1$ is prime for 
all $a\in \mathcal A$ if $q> z^{10z}$ and $z$ 
is large enough. This in
combination with the part 1 of
Proposition \ref{admissiblecounting} then suggests that there are infinitely
many primes $q$ for which 
\begin{equation}
\label{geenzinmeer2}
\sum_{\bfrac{2q+1\le p\le q(\log q)^A}{p\equiv  1\pmod*q}}\frac{1}{p}\ge (1-\epsilon)\frac{L_3}{q}.
\end{equation}
This estimate, together with \eqref{gammaratio} and
the already obtained lower bound \eqref{FLMbound},
then finishes our argumentation.
\par These two speculations taken together imply the
following weaker one.
\begin{Speculation}
There exists a function $g(q)$ such that
$$\lim \inf_{q\rightarrow \infty}\frac{\gamma_q}{g(q)\log q}
=2\lim \inf_{q\rightarrow \infty}\frac{\log r(q)}{g(q)}<0.$$
\end{Speculation}
In case $g(q)$ is not the $\log \log \log$ devil from the title, it
is certainly a close cousin!
\par Comparison of Conjecture \ref{zucht} and Theorem \ref{zucht1} 
suggests that $\gamma_q/\log q$ and
$1-2|\log r(q)|$ behave similarly, which is consistent with the three 
speculations
presented in this section.

\section{Prospect}
\subsection{Polymath} Recent progress on gaps between primes allows one to meet
the challenge below for some
$C>0$.
Indeed, according to James Maynard 
\cite{Maynard}, recent results allow
one to take $C=1/246$.
\begin{Challenge}
\label{challie}
Find a set $\mathcal A=\{a_1,\ldots,a_s\}$ such that
provably for some $B>0$ there are 
$\gg x/\log^B x$ primes $q\le x$ such that
$a_1q+1,\ldots, a_sq+1$ are all prime and, in addition,
$$\sum_{i=1}^s\frac{1}{a_i}\ge C,$$
with $C$ as large as possible.
\end{Challenge}
Conjecturally $C$ can be taken arbitrarily large, 
cf. Theorem \ref{unbounded}.
\begin{prop} 
If one meets the challenge for any $C>2$, then
there are $\gg x/\log^B x$ primes $q\le x$ for
which part 1 of Ihara's conjecture is
false and, moreover, $\gamma_q<(2-C)\log q$.
\end{prop}
\begin{proof}
Similar to that of Theorem \ref{maaah}.
\end{proof}
\subsection{Kummer for arbitrary cyclotomic fields}
It is not difficult to formulate a generalized Kummer
conjecture, where instead of the primes we range
over the integers. 
Goldstein \cite{Gold} established that, as $r$ tends to
infinity and $q$ is fixed, we have
$$\log h_1(q^r)\sim \frac{r}{4}\Big(1-\frac{1}{q}\Big)q^r\log q.$$
Myers \cite{MM} obtained some results along 
the lines of Ram Murty and Petridis \cite{MP}. 
Fouvry \cite{Fouvry} determined the average order of $|\gamma_m|$ 
(see Section \ref{onaverage}).
Quite likely further
results can be obtained, e.g., it is perhaps possible to find
explicit composite integers $m$ for which $\gamma_m<0$.

\subsection{Strengthening the analogy ({\normalfont{Moree and 
Saad Eddin}} \cite{new})} Comparison of
\eqref{gammaq} and \eqref{hasse} suggests that one can expect an even closer
analogy between $r(q)$ and the difference
\begin{equation}
\label{gammaq2}
\gamma_q-\gamma_q^+=\sum_{\chi(-1)=-1}\frac{L'(1,\chi)}{L(1,\chi)},
\end{equation}
which results on subtracting \eqref{EKL} from \eqref{gammaq}.
In particular, it is to be 
expected that $\gamma_q-\gamma_q^+$ will display, like $\log r(q)$, a more
symmetric behaviour around the origin than $\gamma_q$ does. Also
recall that $r(q)$ and $\gamma_q-\gamma_q^+$ both appear 
in the Taylor series \eqref{Taylor}.\\

\par \noindent {\tt Acknowledgment}. 
I like to thank James Maynard for pointing out that one can
take $C=1/246$ in Challenge \ref{challie}. Furthermore, I am 
grateful to 
Alexandru Ciolan, Sumaia 
Saad Eddin and Alisa Sedunova for proofreading and help with
editing an earlier version. Ignazio
Longhi and the referee kindly pointed 
out some disturbing typos.
\par The similarity between
Kummer's and Ihara's conjectures was pointed out
by Andrew Granville after a talk given by
Kevin Ford on \cite{FLM}. At that point the authors 
of \cite{FLM} had independently 
obtained Theorem \ref{unbounded}, but not Granville's Proposition
\ref{admissiblecounting}, the latter being precisely the result 
used by Granville to unleash the $\log \log \log$ devil. Once
at the loose, it created havoc also among  the Euler-Kronecker
constants.


\begin{thebibliography}{99}

\bibitem{AC1} N.C. Ankeny and S. Chowla, The class number of the cyclotomic 
field, {\it Proc. Nat. Acad. Sci. U. S. A.} {\bf 35} (1949), 529--532.

\bibitem{AC2} N.C. Ankeny and S. Chowla, The class number of the cyclotomic 
field, {\it Canadian J. Math.} {\bf 3} (1951), 486--494.

\bibitem{Badz} A.I.~Badzyan,
The Euler--Kronecker constant, {\it Mat. Zametki} {\bf 87} (2010), 45--57.
English Translation in {\it Math. Notes} {\bf 87} (2010), 31--42.

\bibitem{CG} E.S. Croot, III and A. Granville, Unit fractions and 
the class number of a cyclotomic field, 
{\it J. London Math. Soc. (2)} {\bf 66} (2002), 579--591. 

\bibitem{Debaene} K. Debaene, The first factor of the class number of the 
$p$-th cyclotomic field, 
{\it Arch. Math. (Basel)} {\bf 102} (2014), 237--244. 

\bibitem{Erd} P.~Erd\H os, On a problem of S. Golomb\footnote{The title of the
published paper has ``G. Golomb'', a misprint}, 
{\it J. Austral. Math. Soc.} {\bf 2} (1961/1962), 1--8. 


\bibitem{FLM} {K.~Ford, F.~Luca and P.~Moree}, Values of the Euler $\phi$-function not divisible by
a given odd prime, and the distribution of Euler-Kronecker constants for
cyclotomic fields, {\it Math. Comp.} {\bf 83} (2014), 1447--1476.

\bibitem{Fouvry} \'E. Fouvry, Sum of Euler-Kronecker constants over consecutive cyclotomic fields, 
{\it J. Number Theory} {\bf 133} (2013), 1346--1361.

\bibitem{FGW} G. Fung, A. Granville and H.C. Williams, 
Computation of the first factor of the class number of cyclotomic fields, 
{\it J. Number Theory} {\bf 42} (1992), 297--312. 

\bibitem{Gold} L.J. Goldstein, On the class numbers of cyclotomic fields, 
{\it J. Number Theory} {\bf 5} (1973), 58--63. 

\bibitem{Gr} A.~Granville, On the size of the first factor of the class number
of a cyclotomic field, {\it Inv. Math.} {\bf 100} (1990), 321--338.

\bibitem{Hash} Y. Hashimoto, Euler constants of Euler products, 
{\it J. Ramanujan Math. Soc.} {\bf 19} (2004), 1--14. 

\bibitem{HIKW} Y.~Hashimoto,  Y.~Iijima, N.~Kurokawa and M.~Wakayama, Euler's constants for the 
Selberg and the Dedekind zeta functions, {\it Bull. Belg. Math. Soc. Simon Stevin} {\bf 11} (2004), 493--516.


\bibitem{Hasse} H.~Hasse, \"Uber die Klassenzahl Abelscher 
Zahlk\"orper, Mathematische Lehr- b\"ucher und Monographien, Band I, 
Berlin, Akademie-Verlag, 1952.

\bibitem{I} Y.~Ihara, On the Euler-Kronecker constants of global fields and primes
with small norms, in V.~Ginzburg, ed., {\it Algebraic
Geometry and Number Theory: In Honor of Vladimir Drinfeld's 50th Birthday}, Progress in Mathematics {\bf 850}, 
Birkh\"auser Boston, Cambridge, MA, 2006, 407--451.

\bibitem{I2} Y.~Ihara, The Euler-Kronecker invariants in various families of global 
fields, {\it Proc. of AGCT 2005} (Arithmetic Geometry and Coding Theory
10), Ed. F. Rodier et al., S\'eminaires et Congr\`es {\bf 21} (2009), 79--102.

\bibitem{I3} Y.~Ihara, On ``$M$-functions" closely related to the distribution of $L'/L$-values, {\it Publ. Res. Inst. Math. Sci.} 
{\bf 44} (2008), 893--954. 

\bibitem{IKMS} Y.~Ihara, V.K Murty and M.~Shimura, On the logarithmic derivatives of 
Dirichlet $L$-functions at $s=1$, {\it Acta Arith.} {\bf 137} (2009), 253--276. 


\bibitem{Kummer} E.E. Kummer, M\'emoire sur la th\'eorie des nombres complexes compos\'ees de racines de l'unit\'e
et des nombres entiers, {\it J. Math. Pures Appl.} {\bf 16} (1851), 377--498 (1851); Collected Works, Vol. I.,
pp. 363--484.

\bibitem{Lagarias} J.C. Lagarias, Euler's constant: Euler's work and modern developments,
{\it Bull. Amer. Math. Soc. (N.S.)} {\bf 50} (2013), 527--628. 

\bibitem{Lang} S. Lang, {\it Cyclotomic fields I and II}, 
Combined second edition, Graduate Texts in Mathematics
{\bf 121}, Springer-Verlag, New York, 1990.

\bibitem{mamo} J.M.~Masley and H.L.~Montgomery, 
Cyclotomic fields with unique factorization, 
{\it J. Reine Angew. Math.} {\bf 286/287} (1976), 248--256. 

\bibitem{Maynard} J. Maynard, E-mail to author, 04/14/2014.

\bibitem{MV} H.L.~Montgomery and R.C.~Vaughan, The large sieve, 
{\it Mathematika} {\bf 20} (1973), 119--134. 

\bibitem{MIndia} P. Moree, Counting numbers in multiplicative sets: Landau versus Ramanujan,
{\it Mathematics Newsletter} {\bf 21}, no. 3 (2011), 73--81 (arXiv:1110.0708).

\bibitem{new} P. Moree and 
S. Saad Eddin, Euler-Kronecker constants for maximal real cyclotomic fields and
Kummer's conjecture, in preparation.

\bibitem{MKM} M. Mourtada and V.K. Murty, On the Euler Kronecker constant of a cyclotomic field, II. SCHOLAR --a scientific celebration highlighting open lines of arithmetic research, 143--151, 
{\it Contemp. Math.} {\bf 655}, Centre Rech. Math. Proc., Amer. Math. Soc., Providence, RI, 2015. 


\bibitem{MP} M.R. Murty and Y.N. Petridis, On Kummer's 
conjecture, {\it J. Number Theory} {\bf 90} (2001), 294--303. 

\bibitem{KM} V.K. Murty, The Euler-Kronecker constant of a number field, 
{\it Ann. Sci.  Math. Qu\'ebec} {\bf 35}  (2011),  239--247.

\bibitem{MM} M.J.R. Myers, A generalised Kummer's 
conjecture, {\it Glasg. Math. J.} {\bf 52} (2010), 453--472. 

\bibitem{N} W.~Narkiewicz, {\it Elementary and analytic theory of algebraic numbers}. Second edition. 
Springer-Verlag, Berlin; PWN---Polish Scientific Publishers, Warsaw, 1990.




\bibitem{Puchta} J.-C. Puchta, On the class number of 
$p$-th cyclotomic field, 
{\it Arch. Math. (Basel)} {\bf 74} (2000), 266--268. 

\bibitem{shokro} M.A. Shokrollahi, Relative class number of 
imaginary abelian fields of prime conductor below 10000, 
{\it Math. Comp.} {\bf 68} (1999), 1717--1728. 

\bibitem{Siegel} C.L. Siegel, Zu zwei Bemerkungen Kummers, 
{\it Nachr. Akad. Wiss. Göttingen Math.-Phys. Kl.} II 
{\bf 6} (1964), 51--57; Collected Works III, 438--442.

\bibitem{T} M.A.~Tsfasman, Asymptotic behaviour of the Euler-Kronecker constant, in V.~Ginzburg, ed., {\it Algebraic
Geometry and Number Theory: In Honor of Vladimir Drinfeld's 50th Birthday}, Progress in Mathematics, Vol. 850, 
Birkh\"auser Boston, Cambridge, MA, 2006, 453--458.



\bibitem{W} L.C.~Washington, {\it Introduction to cyclotomic fields}, Graduate Texts in Mathematics {\bf 83}, Springer-Verlag, 
New York, 1982.




\end{thebibliography}
\end{document}